\documentclass[12pt]{amsart}
\usepackage[OT4]{fontenc}
\usepackage{graphics,graphicx}
\usepackage{amsmath}
\usepackage{amssymb,amsfonts}
\usepackage {amsthm}
\usepackage{xcolor}

\usepackage{enumerate}
\usepackage{hyperref}
\usepackage[T2A]{fontenc}
\usepackage[cp1251]{inputenc}
\usepackage[all]{xy}
\usepackage[linesnumbered,boxed]{algorithm2e}

\newtheorem{theorem}{Theorem}
\newtheorem{proposition}{Proposition}
\newtheorem{lemma}{Lemma}
\newtheorem{example}{\textbf{Example}}

\makeatletter
\def\blfootnote{\gdef\@thefnmark{}\@footnotetext}
\makeatother

\newcommand\myfootnote[1]{
	\renewcommand{\thefootnote}{}
	\footnotetext{#1}
	\def\thefootnote{\@arabic\c@footnote}
}

\newcommand{\rvline}{\hspace*{-\arraycolsep}\vline\hspace*{-\arraycolsep}}

\renewcommand{\subsection}{\@startsection{subsection}{2}{0mm}{-\baselineskip}{-5pt}{\it \bf}}
\makeatother

\title{Automorphisms and derivations of algebras of infinite matrices}
\author{Oksana Bezushchak}

\begin{document}
	
	\maketitle

\begin{abstract} We describe automorphisms and derivations of several important associative and Lie algebras of infinite matrices over a field.
\end{abstract}

\blfootnote{}               % Keywords
% Math. Subj. Class. codes

\blfootnote{}

\section{Introduction}

Let $\mathbb{F}$ be a ground field and let $A$ be an associative  $\mathbb{F}$-algebra. Recall that an associative  algebra $A$ gives rise to the Lie algebra $A^{(-)}=(A,[a,b] = ab-ba).$ For an arbitrary element $a\in A$ the operator $\text{ad}(a):x \to [a,x]$ is an (\emph{inner}) \emph{derivation} of the associative algebra $A$ and of the Lie algebra $A^{(-)}.$ If $J$ is an ideal of the algebra $A$ then $\text{ad}_J (a)$ denotes the restriction of the derivation $\text{ad}(a)$ to $J.$

An associative algebra $A$ with $1$ gives rise to the group $G(A)$ of invertible elements of $A.$  For an arbitrary invertible element $a\in G(A)$ the conjugation $x \to a^{-1}xa,$ $x\in A,$ is an (\emph{inner}) \emph{automorphism} of the algebra $A$  and, hence, of the algebra $A^{(-)}.$

Let $I$ be an infinite set and let $\mathbb{N},$ $\mathbb{Z}$ denote the set of positive integers  and the set of integers, respectively.

\begin{example}\label{Ex_IN_M_3}  Let $M_{\infty}(I,\mathbb{F})$ be the associative algebra of all $(I\times I)$-matrices over  $\mathbb{F}$ having finitely many nonzero entries $($finitary matrices$).$ The Lie algebras $$\mathfrak{gl}_{\infty}(I,\mathbb{F})  =  M_{\infty}(I,\mathbb{F})^{(-)}, \quad  \quad \mathfrak{sl}_{\infty}(I,\mathbb{F})  =  [  \mathfrak{gl}_{\infty}(I,\mathbb{F}),  \mathfrak{gl}_{\infty}(I,\mathbb{F})  ]$$ and their representation  attracted considerable attention in the literature $($see, \emph{\cite{Frenkel_Penkov_Serganova,Penkov_Serganova})}. \end{example}

A linear transformation $*:A\to A$ of  an associative algebra $A$ is called an \textit{involution} if $(a^*)^* =a,$ $(ab)^* =b^* a^*$ for  arbitrary elements $a,b\in A.$ The subspace of skew-symmetric elements $K(A,^*)=\{a\in A \, | \, a^* =-a\}$ is a subalgebra of the Lie algebra $A^{(-)}.$ The algebra $M_{\infty} (I,\mathbb{F})$ is equipped with the transpose and the symplectic involutions  that give rise to Lie algebras $\mathfrak{o}_{\infty}(I,\mathbb{F}) $ and $\mathfrak{sp}_{\infty}(I,\mathbb{F}) $ of skew-symmetric elements, respectively.

\begin{example}\label{Ex_IN_M_2} Let $M_{rcf}(I,\mathbb{F})$ be the associative algebra of all  $(I\times I)$-matrices over  $\mathbb{F}$ having finitely many nonzero entries in each row  and in each column. It is easy to see that  $M_{\infty}(I,\mathbb{F})$ is an ideal of the algebra $M_{rcf}(I,\mathbb{F}).$ The algebra $M_{rcf}(I,\mathbb{F})$ gives rise to the Lie algebra $$\mathfrak{gl}_{rcf}(I,\mathbb{F})=M_{rcf}(I,\mathbb{F})^{(-)}$$ and to the group $GL_{rcf}(I,\mathbb{F})=G(M_{rcf}(I,\mathbb{F}).$

The transpose and the symplectic involutions on  $M_{\infty}(I,\mathbb{F})$ naturally and uniquely extend to involutions of $M_{rcf}(I,\mathbb{F}).$ \end{example}

\begin{example}\label{Ex_IN_M_1} Let $M(I,\mathbb{F})$ denote the associative algebra of $(I\times I)$-matrices over the field $\mathbb{F}$ having finitely many nonzero entries in each column. If $V$ is a vector space over  $\mathbb{F}$ of dimension $|I|$ then the algebra $\text{End}_{\mathbb{F}}(V)$ of all linear transformations of $V$ is isomorphic to $M(I,\mathbb{F}).$

Let $\mathfrak{gl}(I,\mathbb{F})=M(I,\mathbb{F})^{(-)}.$ The group $G(M(I,\mathbb{F}))$ is isomorphic to the group $GL(V)$ of all invertible linear transformations of $V.$ \end{example}

\begin{example}\label{Ex_IN_M_4} Let $MJ(\mathbb{F})$ be the algebra of Jacobi matrices, that is, $(\mathbb{Z}\times \mathbb{Z})$-matrices having finitely many nonzero diagonals. In other worlds, a matrix $a=(a_{ij})_{i,j\in \mathbb{Z}}$	lies in $MJ(\mathbb{F})$ if there exists $n \in \mathbb{N}$ such that $a_{ij}=0$ whenever $|i-j|>n.$
	
The Lie algebras  $\mathfrak{gl}_J(\mathbb{F})=MJ(\mathbb{F})^{(-)}$ found applications  in the theory of solitones \emph{(\cite{Date_Jimbo_Kashiwara_Miwa,Feigin_Tsygan})}. \end{example}

\begin{example}\label{Ex_IN_M_5} Let  $\text{End}_{fin}(V)$ be the subalgebra of $\text{End}_{\mathbb{F}}(V)$ that consists of all linear transformations of finite range. \end{example}

N.~Jacobson (\cite{Jacobson}, Chap. 9, Sec. 11, Th. 7) showed that for an arbitrary associative algebra $A,$  $M_{\infty}(I,\mathbb{F}) \subseteq A \subseteq M(I,\mathbb{F}),$ every automorphism  $\varphi$ of $A$ is a conjugation by an invertible element from $M(I,\mathbb{F}),$ $\varphi(a)= x^{-1}ax,$ $a\in A.$

K.-H.~Neeb \cite{Neeb} and N.~Stumme \cite{Stumme} described derivations and automorphisms of Lie algebras $\mathfrak{sl}_{\infty}(I,\mathbb{F}), $  $\mathfrak{o}_{\infty}(I,\mathbb{F}), $ $\mathfrak{sp}_{\infty}(I,\mathbb{F}) $ over a field of zero characteristic.

\section{Main results}

Theorems \ref{Th_IN_M_1_NEW} and  \ref{Th_IN_M_3_NEW} below describe all derivations   of the associative $\mathbb{F}$-algebras $M_{\infty}(I,\mathbb{F}),$ $\text{End}_{fin}(V),$ $M_{rcf}(I,\mathbb{F}),$ $MJ(\mathbb{F}),$ $M(I,\mathbb{F})$ (Theorem \ref{Th_IN_M_1_NEW})  and all  automorphisms of the associative $\mathbb{F}$-algebras $M_{\infty}(I,\mathbb{F}),$ $\text{End}_{fin}(V),$ $M_{rcf}(I,\mathbb{F}),$ $M(I,\mathbb{F})$ (Theorem \ref{Th_IN_M_3_NEW}). Theorem \ref{Th_IN_M_4_NEW} describes anti-automorphisms of the algebras $M_{\infty}(I,\mathbb{F}),$ $M_{rcf}(I,\mathbb{F}),$ $MJ(\mathbb{F})$ and shows that the algebras  $\text{End}_{fin}(V),$ $M(I,\mathbb{F})$ do not have any  anti-automorphisms.

Finally, using the proof of Herstein's conjectures  \cite{Herstein} by K.I.~Beidar, M.~Bre\v{s}ar, M.~Chebotar and W.S.~Martindale (see, \cite{Bei_Bre_Cheb_Mart_1,Bei_Bre_Cheb_Mart_2,Bei_Bre_Cheb_Mart_3}) we describe derivations of Lie algebras $\mathfrak{sl}_{\infty}(I,\mathbb{F}), $  $\mathfrak{o}_{\infty}(I,\mathbb{F}), $ $\mathfrak{sp}_{\infty}(I,\mathbb{F}) ,$ $\mathfrak{gl}_{rcf}(I,\mathbb{F}), $  $\mathfrak{gl}_{J}(\mathbb{F}) $  in Theorem \ref{Th_IN_M_2_NEW} and automorphisms of Lie algebras $\mathfrak{sl}_{\infty}(I,\mathbb{F}), $  $\mathfrak{o}_{\infty}(I,\mathbb{F}), $ $\mathfrak{sp}_{\infty}(I,\mathbb{F}) ,$ $\mathfrak{gl}_{rcf}(I,\mathbb{F}), $  $\text{End}_{fin}(V)^{(-)},$ $\mathfrak{gl}(I,\mathbb{F})$ in Theorem \ref{Th_IN_M_5_NEW}.

When dealing with Lie algebras of  infinite matrices we use the following theorem which is of some independent interest.

\begin{theorem}\label{Pr_1_IN_M} For an arbitrary infinite set $I$ we have:
\begin{enumerate}
  \item[\emph{(a)}] $M(I,\mathbb{F})=[M(I,\mathbb{F}),M(I,\mathbb{F})],$
  \item[\emph{(b)}] $M_{rcf}(I,\mathbb{F})=[M_{rcf}(I,\mathbb{F}),M_{rcf}(I,\mathbb{F})],$
  \item[\emph{(c)}] $MJ(\mathbb{F})=[MJ(\mathbb{F}),MJ(\mathbb{F})],$
  \item[\emph{(d)}] $\mathfrak{o}_{\infty}(I,\mathbb{F})=[\mathfrak{o}_{\infty}(I,\mathbb{F}),\mathfrak{o}_{\infty}(I,\mathbb{F})],$ \\ $\mathfrak{sp}_{\infty}(I,\mathbb{F})=[\mathfrak{sp}_{\infty}(I,\mathbb{F}),\mathfrak{sp}_{\infty}(I,\mathbb{F})].$
\end{enumerate}
\end{theorem}

In particular, the algebras $\mathfrak{gl}(I,\mathbb{F}),$ $\mathfrak{gl}_{rcf}(I,\mathbb{F})$ and $\mathfrak{gl}_{J}(\mathbb{F})$ are perfect.

\begin{theorem}\label{Th_IN_M_1_NEW}  \begin{enumerate}
		\item[\emph{(a)}]  An arbitrary derivation of the algebra $M_{\infty}(I, \mathbb{F})$ is of the type  $$\emph{ad}_{M_{\infty}(I,\mathbb{F})} (a), \quad \text{where} \quad a\in M_{rcf}(I,\mathbb{F});$$
		\item[\emph{(b)}]   an arbitrary derivation of the algebra $\emph{End}_{fin}(V)$ is of the type $$\emph{ad}_{\emph{End}_{fin}(V)} (a),\quad \text{where} \quad a\in \emph{End}_{\mathbb{F}}(V);$$
		\item[\emph{(c)}] all derivations of the algebras $M_{rcf}(I, \mathbb{F}),$  $MJ(\mathbb{F}),$ $M(I, \mathbb{F})$ are inner.
\end{enumerate}		
\end{theorem}

\begin{theorem}\label{Th_IN_M_2_NEW} Let $\mathbb{F}$ be a field of the characteristic not equal to $2.$ \begin{enumerate}
		\item[\emph{(a)}]  An arbitrary derivation of the Lie algebra $\mathfrak{sl}_{\infty}(I,\mathbb{F}) $ is of the type $$\emph{ad}_{\mathfrak{sl}_{\infty}(I,\mathbb{F}) } (a),\quad \text{where} \quad a\in \mathfrak{gl}_{rcf}(I,\mathbb{F});$$
			\item[\emph{(b)}]  an \ arbitrary \ derivation \ of \ the \ Lie \ algebra \ $\mathfrak{o}_{\infty}(I,\mathbb{F}) $ \ $($resp. \\ $\mathfrak{sp}_{\infty}(I,\mathbb{F}) )$ is of the type $$\emph{ad}_{\mathfrak{o}_{\infty}(I,\mathbb{F}) } (a) \quad (\text{resp.} \quad \emph{ad}_{\mathfrak{sp}_{\infty}(I,\mathbb{F}) } (a)),$$ where $ a\in K(M_{rcf}(I,\mathbb{F}),t)$ $($resp. $K(M_{rcf}(I,\mathbb{F}),s);$
				\item[\emph{(c)}]  all derivations of  Lie algebras $\mathfrak{gl}(I,\mathbb{F}) ,$ $\mathfrak{gl}_{rcf}(I,\mathbb{F}) ,$ $\mathfrak{gl}_{J}(\mathbb{F}) $ are inner.
	\end{enumerate}
\end{theorem}

\begin{theorem}\label{Th_IN_M_3_NEW} \begin{enumerate}
	\item[\emph{(a)}]  An \ arbitrary \ automorphism \ $\varphi$ \ of \ the \ algebra \\ $M_{\infty}(I, \mathbb{F})$ is a conjugation by an element from $GL_{rcf}(I, \mathbb{F}).$ In other words, there exists an element $x\in GL_{rcf}(I, \mathbb{F})$ such that $\varphi(a)=x^{-1}ax$ for all $a\in M_{\infty}(I, \mathbb{F});$
	\item[\emph{(b)}]   an arbitrary automorphism of the algebra $\emph{End}_{fin}(V)$ is a conjugation by an element from $GL(V);$
	\item[\emph{(c)}] all automorphisms of the algebras $M_{rcf}(I, \mathbb{F}),$   $M(I, \mathbb{F})$ are inner.
\end{enumerate} \end{theorem}

Given an algebra $A$ an invertible linear transformation $\varphi:A\to A$ is called an \emph{anti-automorphism} if $\varphi(ab)= \varphi(b)\varphi(a)$ for arbitrary elements $a,b\in A.$ It is easy to see that an  anti-automorphism is an isomorphism $A\to A^{op}$ with the opposite algebra $ A^{op}=(A,a*b=ba).$ The transpose transformations  $(a_{ij})^t =(a_{ji}),$ $ a_{ij} \in  \mathbb{F},$ are anti-isomorphisms of the algebras $M_{\infty}(I, \mathbb{F}),$ $M_{rcf}(I, \mathbb{F}),$ $MJ(\mathbb{F}).$

\begin{theorem}\label{Th_IN_M_4_NEW} \begin{enumerate}
		\item[\emph{(a)}] An arbitrary anti-automorphism  of one of the algebras $M_{\infty}(I, \mathbb{F}),$ $M_{rcf}(I, \mathbb{F}),$ $MJ(\mathbb{F})$ is a composition of the  transpose and an automorphism $($see, Theorem $\ref{Th_IN_M_3_NEW});$
	\item[\emph{(b)}] algebras $\emph{End}_{fin}(V)$ and $M(I, \mathbb{F})$ do not have anti-isomorphisms. In other words, these algebras are not isomorphic to their opposite algebras.
\end{enumerate} \end{theorem}

\begin{theorem}\label{Th_IN_M_5_NEW} Let $\mathbb{F}$ be a field of the characteristic not equal to $2.$  \begin{enumerate}
		\item[\emph{(a)}]  An arbitrary automorphism $\varphi$ of the Lie algebra $L=\mathfrak{sl}_{\infty}(I,\mathbb{F}) $ is of the type $\varphi(a)=x^{-1}ax,$ $a\in L,$ or of the type $\varphi(a)=-x^{-1}a^t x,$ $a\in L,$  where $x\in GL_{rcf}(I,\mathbb{F});$
		\item[\emph{(b)}] an arbitrary automorphism $\varphi$ of the Lie algebra $L=\mathfrak{o}_{\infty}(I,\mathbb{F}) $ $($resp. $\mathfrak{sp}_{\infty}(I,\mathbb{F}))$ is of the type $\varphi(a)=x^{-1}ax,$ $a\in L,$ where $x\in GL_{rcf}(I,\mathbb{F})$ and $x x^t \in\mathbb{F}$ $($resp. $x x^s \in\mathbb{F});$
				\item[\emph{(c)}]  an arbitrary automorphism $\varphi$ of the Lie algebra $\mathfrak{gl}_{rcf}(I,\mathbb{F}) $ is of the type $\varphi(a)=x^{-1}ax,$ $a\in \mathfrak{gl}_{rcf}(I,\mathbb{F}) ,$ or of the type $\varphi(a)=-x^{-1}a^t x,$ $a\in \mathfrak{gl}_{rcf}(I,\mathbb{F}) ,$  where $x\in G(M_{rcf}(I, \mathbb{F}));$
		\item[\emph{(d)}] an arbitrary automorphism $\varphi$ of the Lie algebra $\mathfrak{gl}(I,\mathbb{F}) $ is of the type $\varphi(a)=x^{-1}ax,$ $a\in \mathfrak{gl}(I,\mathbb{F}),$   where $x\in G(M(I,\mathbb{F}));$
\item[\emph{(e)}]  an arbitrary automorphism $\varphi$ of the Lie algebra $\emph{End}_{fin}(V)^{(-)}$ is of the type $\varphi(a)=x^{-1}ax,$ $a\in \emph{End}_{fin}(V)^{(-)},$   where $x\in G(\emph{End}_{fin}(V)).$
			\end{enumerate}
\end{theorem}

\textbf{Remark}. We don't have a description of automorphisms of the algebra~$\mathfrak{gl}_{J}(\mathbb{F}) .$

\section{Derivations of associative algebras of infinite matrices}

Recall the definition of the Tykhonoff topology. Let $X,$ $Y$ be arbitrary sets. Let $\text{Map}(X,Y)$ be the set of mappings $X\rightarrow Y.$ For distinct elements $a_1, \ldots, a_n\in X$ and arbitrary elements  $b_1, \ldots, b_n\in Y,$ $n \geq 1,$  consider the subset $$M(a_1, \ldots, a_n; b_1, \ldots, b_n)=\{ f:X \to Y \, | \, f(a_i)=b_i, \ 1 \leq i \leq n\}$$ of $\text{Map}(X,Y).$ The Tykhonoff topology on $\text{Map}(X,Y)$ is generated by all open sets of this type. In other words, $\text{Map}(X,Y)$ as a topological space is homeomorphic to the space $Y^{|X|},$ the Tykhonoff product of $|X|$ copies of $Y,$ where $Y$ is equipped with the discrete topology.

An associative $\mathbb{F}$-algebra $A$ is called a \emph{locally matrix algebra} if for each finite subset of $A$ there exists a subalgebra $B\subset A$ containing this subset and isomorphic to the algebra $M_n(\mathbb{F})$ of $(n\times n)$-matrices over $\mathbb{F}$ for some positive integer $n.$ For properties and  theory of locally matrix algebras, see \cite{Baranov2,14,BezushchakCarp,BezOl,BezOl_2,Kurosh}.

If $A$ is an associative $\mathbb{F}$-algebra and $M$ is an $A$-bimodule then the space $C^1(A,M)$ of all \emph{bimodule derivations} $A \to M$ lies in $\text{Map}(X,Y)$ and, therefore, is equipped with the Tykhonoff topology. Let $B^1 (A,M)$ be the  space of all \emph{inner bimodule derivations} $A\rightarrow M.$ The factor-space $H^1 (A,M)=C^1 (A,M)/B^1 (A,M)$ is called the first cohomology space (for details see, \cite{Drozd_Kirichenko,Pierce}).

We will need the following generalization of Theorem 1(1) from \cite{14}.

\begin{lemma}\label{Lem1_IN_M} If $A$ is a locally matrix algebra then for an arbitrary $A$-bimodule $M$ the subspace $B^1(A,M)$  is dense  in the space $C^1(A,M)$ in the Tykhonoff topology. \end{lemma}

\begin{proof} We need to show that for arbitrary elements $a_1,\ldots, a_n \in A$ and an arbitrary derivation $d:A \to M$ there exists an element $x\in M$ such that $[x,a_i]=d(a_i),$ $1 \leq i \leq n.$

There exists a subalgebra $A_1 \subset A$ such that $a_1,\ldots, a_n \in A_1$ and $A_1\cong M_k(\mathbb{F}).$ Then $$ M_1 = d(A_1)+ A_1 d(A_1)+ d(A_1)A_1 + A_1 d(A_1) A_1$$ is a finite-dimensional $A_1$-submodule of $M$ and the restriction of the derivation $d$ to $A_1$ is a derivation $A_1 \to M_1.$ Since every bimodule derivation of a matrix algebra over a field is inner it follows that there exists an element $x\in M_1$ such that $d(a)=[x,a]$ for all elements $a\in A_1.$  This completes the  proof of the lemma. \end{proof}

Let $(I\times I)(\mathbb{F})$ denote the vector space of all $(I\times I)$-matrices over $\mathbb{F}.$ Since for arbitrary matrices $a\in (I\times I)(\mathbb{F}),$ $b\in M_{\infty}(\mathbb{F})$ the products $ab,$ $ba$ are well defined it follows that the vector space $(I\times I)(\mathbb{F})$ is a bimodule over $M_{\infty}(\mathbb{F}).$ We will start with the following proposition.

\begin{proposition}\label{Th_IN_M_1(1)} $H^1 (M_{\infty}(I,\mathbb{F}), (I\times I)(\mathbb{F}) )= (0).$
\end{proposition}

\begin{proof} Let $d: M_{\infty}(I,\mathbb{F}) \to (I\times I)(\mathbb{F})$ be a bimodule derivation. By Lemma \ref{Lem1_IN_M}, for an arbitrary nonempty finite subset $J\subset I$ there exists an element $y_J \in (I\times I)(\mathbb{F})$ such that $d(a) = [y_J ,a]$ for all elements $a\in M_{|J|}(\mathbb{F}).$
	
Divide the matrix $y_J$ into blocks
	$$ y_J =\begin{pmatrix}
		y_J (11) &    y_J (12) \\
		y_J (21) &   y_J (22)
	\end{pmatrix},$$
where $y_J (11) \in M_{|J|}(\mathbb{F});$ $y_J (12)$ is a $(J\times (I\, \diagdown J))$-matrix over  $\mathbb{F};$ $y_J (21)$ is a $( (I\, \diagdown J)\times J)$-matrix over  $\mathbb{F};$ and $y_J (22)$ is a $((I\, \diagdown J)\times (I \,  \diagdown  J))$-matrix over  $\mathbb{F}.$ For an arbitrary element $a\in M_{|J|}(\mathbb{F})$ we have
	$$ [y_J ,a] =\begin{pmatrix}
	[y_J (11),a] &  -a\, y_J (12) \\
	y_J (21)\, a &  0
	\end{pmatrix}.$$

Let $J_1$  be another finite subset of $I,$ $J \subseteq J_1 .$ Then $[y_{J} , a]=[y_{J_1} , a]$ for all elements $a\in M_{|J|}(\mathbb{F}).$ Hence the $(J\times J)$-minor of $y_{J_1}$ differs from $y_J (11)$ by a scalar matrix.

We define a $(I\times I)$-matrix $y$ as follows. Fix $i_0 \in I.$ Let $y_{i_0,i_0}=0.$  All matrices $y_J,$ where $J$ runs over finite subsets of $I,$ $i_0 \in J,$ can be selected so that  $(y_J)_{i_0,i_0}=0.$

For $i,j\in I$ choose a finite subset $J\subset I$ such that $i,j,i_0 \in J.$ Define $y_{ij}= (y_J)_{ij}.$ Let $J_1$ be a finite subset of $I,$ $J\subseteq J_1.$ Since $y_J(11)$ and the $(J\times J)$-minor of $y_{J_1}$ differ by a scalar and $ (y_J)_{i_0,i_0}=(y_{J_1})_{i_0,i_0}=0$ it follows that $ (y_J)_{ij}=(y_{J_1})_{ij}.$ Hence $y_{ij}$ does not depend on a choice of the subset $J.$

For an arbitrary element $a\in M_{\infty}(I,\mathbb{F})$ the $(J\times J)$-minors of $(I\times I)$-matrices $d(a)=[y_J,a]$ and $[y,a]$ coincide as long as  $J \subset I$ is a finite subset, $i_0\in J,$ and $a\in M_{|J|}(\mathbb{F}).$ This implies that $d(a)= [y,a]$ and completes the proof of  the proposition.
\end{proof}

\begin{lemma}\label{Lem2_IN_M} Let $y\in (I\times I)(\mathbb{F}).$
\begin{enumerate}
	\item[\emph{(a)}] The inclusion $[y,M_{\infty}(I,\mathbb{F})]\subseteq M(I,\mathbb{F})$ implies $y\in M(I,\mathbb{F}).$
	\item[\emph{(b)}] The inclusion $[y,M_{\infty}(I,\mathbb{F})]\subseteq M_{rcf}(I,\mathbb{F})$ implies $y\in M_{rcf}(\mathbb{F}).$
\end{enumerate}
 \end{lemma}

\begin{proof}  Let $e_{pq},$ $p,q \in I,$ denote the matrix unit having $1$ at the position $(p,q)$ and zeros elsewhere. We have $$ [y,e_{jj}]=y\, e_{jj} - e_{jj} \, y, \quad [y,e_{jj}]_{ij}= \begin{cases} y_{ij}, & i\neq j ,\\
0, & i=j.	
	\end{cases} $$ If $[y,e_{jj}]\in M(I,\mathbb{F})$ then $y_{ij}\neq 0$ for only finitely many $i\in I.$  Hence $y\in M(I,\mathbb{F}).$ Similarly, $$[y,e_{ii}]_{ij}= \begin{cases} -y_{ij}, & i\neq j, \\
	0, & i=j.	
\end{cases} $$ If $[y,e_{ii}]\in M_{rcf}(I,\mathbb{F})$ then $y_{ij}\neq 0$ for only finitely many $j\in I.$ This completes the proof of the lemma. \end{proof}

\begin{proof}[Proof of Theorem $\ref{Th_IN_M_1_NEW}$] (a) Let $d: M_{\infty}(I,\mathbb{F}) \rightarrow M(I,\mathbb{F})$ be a bimodule derivation. By Proposition \ref{Th_IN_M_1(1)}, there exists an $(I\times I)$-matrix $y=(y_{ij})_{i,j\in I}$ such that $d(a)=[y,a]$ for an arbitrary element $a\in M_{\infty}(I,\mathbb{F}).$ By Lemma \ref{Lem2_IN_M}(a), the matrix $y$ lies in $M(I,\mathbb{F}).$ Hence \begin{equation}\label{EQ1}
H^1 (M_{\infty}(I,\mathbb{F}), M (I,\mathbb{F}) )= (0).
\end{equation}

 Let $d: M_{\infty}(I,\mathbb{F}) \rightarrow M_{\infty}(I,\mathbb{F})$ be a derivation. By the above, there exists an $(I\times I)$-matrix $y$ such that $d(a)=[y,a]$ for an arbitrary element $a\in M_{\infty}(I,\mathbb{F}).$ By Lemma \ref{Lem2_IN_M}(b), we have  $y\in M_{rcf}(I,\mathbb{F}).$

(b) Choose a basis $v_i,$ $i\in I,$ in the vector space $V.$ Linear transformations from $\text{End}_{fin}(V)$ have matrices lying in the subalgebra $M_{r-fin}(I,\mathbb{F})$  of $M(I,\mathbb{F})$ that consists of $(I\times I)$-matrices that are column-finite and have finite range. These two conditions are equivalent to matrices from $M_{r-fin}(I,\mathbb{F})$ having finitely many nonzero rows.

We have $$M_{\infty}(I,\mathbb{F})  \subset M_{r-fin}(I,\mathbb{F}) \lhd M(I,\mathbb{F}).$$ Let $d: M_{r-fin}(I,\mathbb{F}) \rightarrow M_{r-fin}(I,\mathbb{F})$ be a derivation. The restriction of $d$ to $M_{\infty}(I,\mathbb{F})$ is a bimodule derivation from $M_{\infty}(I,\mathbb{F})$ to $M(I,\mathbb{F}).$ By (\ref{EQ1}),  there exists a matrix $y\in M(I,\mathbb{F})$ such that $d(a)=[y,a]$ for an arbitrary element $a\in M_{\infty}(I,\mathbb{F}).$

Consider the derivation $$d\,'=d-\text{ad}_{M_{r-fin}(I,\mathbb{F})}(y) $$ of the algebra $M_{r-fin}(I,\mathbb{F}).$ We have $d\,'( M_{\infty}(I,\mathbb{F})) =(0).$ For an  arbitrary element $a\in M_{r-fin}(I,\mathbb{F})$ and arbitrary indices $i,j\in I$ we have $$d\,'(e_{ii} \, a \, e_{jj})=e_{ii}\, d\,'(a)\,  e_{jj} =0.$$ Hence, $d\,'(a)_{ij}=0.$ We showed that $d\,' =0,$ which completes the proof of the part (b).

(c) Let $d: M_{rcf}(I,\mathbb{F}) \rightarrow M_{rcf}(I,\mathbb{F})$ be a  derivation. There exists a matrix $y\in M(I,\mathbb{F})$ such that $d(a)=[y,a]$ for an arbitrary element $a\in M_{\infty}(I,\mathbb{F}).$ By Lemma \ref{Lem2_IN_M}, the inclusion $$[y, M_{\infty}(I,\mathbb{F})] \subseteq M_{rcf}(I,\mathbb{F})$$ implies $y\in M_{rcf}(I,\mathbb{F}).$ Consider the derivation $$d\,'=d-\text{ad}(y) \ \text{of} \ M_{rcf}(I,\mathbb{F}), \quad d\,'(M_{\infty}(I,\mathbb{F}))=(0).$$ As in the proof of Theorem \ref{Th_IN_M_1_NEW}$(b),$ for  arbitrary indices $i,j \in I$ we have $$d\,'(e_{ii} \, M_{rcf}(I,\mathbb{F}) \, e_{jj})=e_{ii}\, d\,'(M_{rcf}(I,\mathbb{F})) \, e_{jj} =0$$ which implies $d\,'=0,$ $d=\text{ad}(y).$

 Let $d: MJ(\mathbb{F}) \rightarrow MJ(\mathbb{F})$ be a  derivation. Since $$ M_{\infty}(\mathbb{Z},\mathbb{F}) \subset MJ(\mathbb{F}) \subset M_{rcf}(\mathbb{Z},\mathbb{F}),$$ $$ H^1 (M_{\infty}(I,\mathbb{F}), M (I,\mathbb{F}) )= (0)$$ by Lemma \ref{Lem2_IN_M}, there exists a matrix $y\in M_{rcf}(\mathbb{Z},\mathbb{F})$ such that $d(a)=[y,a]$ for an arbitrary element $a\in M_{\infty}(\mathbb{Z},\mathbb{F}).$  Consider the bimodule derivation $$d\,'=d-\text{ad}(y) , \quad d\,':MJ(\mathbb{F}) \rightarrow M_{rcf}(\mathbb{Z},\mathbb{F}), \quad d\,'(M_{\infty}(\mathbb{Z},\mathbb{F}) )=(0).$$ Let us show that $d\,'(MJ(\mathbb{F}))=(0).$ Let $a\in MJ(\mathbb{F}),$ $a=(a_{ij})_{i,j\in \mathbb{Z}},$ $a_{ij}\in \mathbb{F}.$ There exists $k\geq 1$ such that $a_{ij}=0$ whenever $|i-j|>k.$

Let $n \geq 1.$ Consider matrices $a\,'(n),$ $a\,''(n):$ $$ a\,'(n)_{ij}= \begin{cases} a_{ij} & \text{if} \ |i|, |j| \leq n  \\
0  & \text{otherwise},	
	\end{cases} \quad \quad a\,''(n)_{ij}= \begin{cases} a_{ij} & \text{if} \ |i|>n \ \text{or} \ |j|> n \\
0 & \text{otherwise}.	
	\end{cases}$$ Clearly, $a=a\,'(n)+a\,''(n).$ If $a\,''(n+k)_{ij}\neq 0$ then $|i|>n$ and $|j|>n,$ that is, $$a\,''(n+k)\in  \begin{pmatrix}
		\ 0_{n\times n} \ & \rvline  & \  0 \ \\ \hline
		\ 0 \ & \rvline  & \ * \
	\end{pmatrix},$$ where $0_{n\times n}$  denote zero $(n\times n)$-matrix. Let us show that $$d\,'(a\,''(n+k))\in  \begin{pmatrix}
		\ 0_{n\times n} \ & \rvline  & \  0 \ \\ \hline
		\ 0 \ & \rvline  & \ * \
	\end{pmatrix}$$ as well. Indeed, if $i\leq n$ or $j\leq n$ then $$ d\,'(a\,''(n+k))_{ij}=d\,'(e_{ii}\, a\,''(n+k) \, e_{jj})=0$$ since $e_{ii}\, a\,''(n+k) \, e_{jj}=0.$ We have $d\,'(a\,'(n+k))=0.$ Hence  $$d\,'(a)\in  \begin{pmatrix}
		\ 0_{n\times n} \ & \rvline  & \  0 \ \\ \hline
		\ 0 \ & \rvline  & \ * \
	\end{pmatrix}$$ for any $n\geq 1.$ Hence $d\,'(a)=0,$ $d(a)=[y,a]$ for an arbitrary $a\in MJ(\mathbb{F}).$

Consider the matrix $$ E_1 =\sum_{i\in \mathbb{Z}} e_{i,i+1} \in MJ(\mathbb{F}).$$ It is easy to see that if $y \not\in MJ(\mathbb{F})$ then $[y,E_1] \not\in MJ(\mathbb{F})$ as well. This implies that $y\in MJ(\mathbb{F}).$

Finally, let $d: M(I,\mathbb{F})\rightarrow M(I,\mathbb{F})$ be a derivation. The ideal $M_{r-fin}(I,\mathbb{F})$ is invariant with respect to all derivations. By Theorem \ref{Th_IN_M_1_NEW}$(b),$ there exists an element $y\in M(I,\mathbb{F})$ such that $$ d\,|\,_{M_{r-fin}(I,\mathbb{F})} = \text{ad}_{M_{r-fin}(I,\mathbb{F})}(y).$$ Let $d\,'=d-\text{ad}(y).$ Then $d\,'(M_{r-fin}(I,\mathbb{F}))=(0).$ Arguing as above, we get $d\,' =0.$ This completes the proof of Theorem \ref{Th_IN_M_1_NEW}. \end{proof}

\section{Lie algebras of infinite matrices}\label{Lie algebras}

We will start with the proof of Theorem \ref{Pr_1_IN_M}.

Let $I$ be an arbitrary infinite set. It follows that $|I|=|\mathbb{N}\times I|.$ That is why without loss of generality, we will assume that $I$ is a direct product of $\mathbb{N}$ with another set $J,$ $|J|=|I|,$ $I=\mathbb{N}\times J.$

Therefore, we can view an arbitrary $(I\times I)$-matrix as an $(\mathbb{N}\times \mathbb{N})$-matrix over $(J\times J)(\mathbb{F}).$ For two positive integers $i,j\in \mathbb{N}$ and a $(J\times J)$-matrix $x$ let $e_{ij}(x)$ denote an $(I\times I)$-matrix having the block $x$ at the position $(i,j)$ and zeros elsewhere.

Consider the matrix $$ E =\sum_{i=1}^{\infty} e_{i,i+1}(\text{Id}) \in (I\times I)(\mathbb{F}),$$ where $\text{Id}$ is the identity $(J\times J)$-matrix. The $(I\times I)$-matrix $E$ contains $\leq 1$ nonzero elements in each row and in each column. Therefore, for an arbitrary $(I\times I)$-matrix $x$ the products $Ex,$ $xE$ make sense.

Let $a=(a_{ij})_{i,j\in \mathbb{N}},$ $a_{ij}\in (J\times J)(\mathbb{F}).$ Consider the $(I\times I)$-matrix
\begin{equation}\label{eq_1_part1}
\tilde{a}=(\tilde{a}_{ij})_{i,j\in \mathbb{N}}, \quad  \tilde{a}_{ij}=\sum_{k=0}^{\infty} a_{i-1-k,j-k},
 \end{equation}
where we let $a_{ij}=0$ for $i\leq 0$ or $j\leq 0.$ It is easy to see that for $a\in M(I,\mathbb{F})$ the matrix $\tilde{a}$ also lies in $M(I,\mathbb{F}).$

\begin{lemma}\label{lem_0}\!\footnote{V.V.~Sergeichuk (2020), personal communication.} $[E,\widetilde{a}]=a.$ \end{lemma}

\begin{proof} We have $$  E\  \widetilde{a}=\Big( \sum _{i=1}^{\infty}e_{i,i+1}(1) \Big) \Big(\sum _{j,k=1}^{\infty}e_{jk}(\widetilde{a}_{jk})\Big)
	=\sum _{i,k=1}^{\infty}e_{ik}(\widetilde{a}_{i+1,k});$$
\begin{multline*}    \widetilde{a}\ E=\Big( \sum _{j,k=1}^{\infty}e_{jk}(\widetilde{a}_{jk}) \Big) \Big(\sum _{i=1}^{\infty}e_{i,i+1}(1)\Big) = \\
=\sum _{j,i=1}^{\infty}e_{j,i+1}(\widetilde{a}_{ji})=\sum _{i\geq 1, k\geq 2}e_{ik}(\widetilde{a}_{i,k-1}).\end{multline*}
Finally, $$  [E,  \widetilde{a}]= \sum_{i,k=1}^{\infty}e_{ik}(\widetilde{a}_{i+1,k}-\widetilde{a}_{i,k-1}),$$ where $\widetilde{a}_{i,0}=0.$ It remains to verify that
\begin{equation}\label{eq_2_part1} \widetilde{a}_{i+1,j}-\widetilde{a}_{i,j-1}=a_{ij} \end{equation}
for all $i,j\in\mathbb{N}. $

Define $a_{ij}=0$ whenever $i\leq 0$ or $j\leq 0.$ Then  $$ \widetilde{a}_{ij}= \sum_{k=0}^{\infty}a_{i-1-k,j-k}.$$
Now, $$ \widetilde{a}_{i+1,j} - \widetilde{a}_{i,j-1}= \sum_{k=0}^{\infty}a_{i-k,j-k}- \sum_{k=0}^{\infty}a_{i-1-k,j-1-k}=a_{ij},$$ which completes the proof of the lemma. \end{proof}

\begin{proof}[Proof of Theorem $\ref{Pr_1_IN_M}$] (a) An arbitrary $(I\times I)$-matrix $a$ over the field $\mathbb{F}$ can be divided into blocks $a=(a_{ij})_{\mathbb{N}\times\mathbb{N}},$ where each block $a_{ij}$ is a  $(J\times J)$-matrix.
If $a$ lies in  $\mathfrak{gl}(I,\mathbb{F})$ then each block $a_{ij}$ lies in  $\mathfrak{gl}(J,\mathbb{F}).$  The reverse statement is not true.

Suppose that $a \in \mathfrak{gl}(I,\mathbb{F}).$ We will show that the matrix $\widetilde{a}$ (see above) also lies in $\mathfrak{gl}(I,\mathbb{F}).$ Choose a column indexed by $(n,\alpha)\in \mathbb{N}\times J,$ $n\in \mathbb{N}, $ $\alpha \in J.$ We need to verify that the $\alpha$-th column of the $(I\times J)$-matrix
$$\begin{pmatrix}
	\widetilde{a}_{1n} \\
    \widetilde{a}_{2n}\\
	\vdots
	\end{pmatrix} $$
contains finitely many nonzero entries. Since the $(I\times I)$-matrix $a$ has finitely many nonzero entries in each column it follows that for an arbitrary $l\in \mathbb{N}$ there exists a positive integer $N(l,\alpha)$ such that the $(J\times J)$-matrix $a_{kl}$ has zero  $\alpha$-th column for $k>N(l,\alpha).$

Let $i>n+\max (N(1,\alpha) + \cdots + N(n,\alpha))=s.$ Then $i-2 \geq n-1.$ By the equality (\ref{eq_1_part1}), we have $$\widetilde{a}_{in} =a_{i-1,n} + \cdots + a_{i-n,1}.$$ All summands on the right hand side have zero $\alpha$-th column. Hence the matrix $\widetilde{a}_{in}$ has zero $\alpha$-th column. The matrix
$$\begin{pmatrix}
	\widetilde{a}_{1n} \\
	\vdots \\
	\widetilde{a}_{sn}
\end{pmatrix} $$
has finitely many nonzero entries in the $\alpha$-th column since the $(J\times J)$-matrices $\widetilde{a}_{1n} ,$ $\dots,$ $\widetilde{a}_{sn} $ have finitely many nonzero entries in each column.

We showed that the matrix $\widetilde{a}$ lies in $\mathfrak{gl}(I,\mathbb{F}).$ As above, let $$E= \sum _{i=1}^{\infty}e_{i,i+1}(\text{Id}_J),$$ where  $\text{Id}_J$ is the identity  $(J\times J)$-matrix. By Lemma \ref{lem_0}, we have $$a=[E,\widetilde{a}] \in [\mathfrak{gl}(I,\mathbb{F}),\mathfrak{gl}(I,\mathbb{F})].$$ Hence, we proved Theorem~\ref{Pr_1_IN_M}(a).

\vspace{10pt}
(b) We need to show that for an arbitrary matrix $a\in M_{rcf}(I,\mathbb{F})$ the matrix $\tilde{a}$ also lies in $M_{rcf}(I,\mathbb{F}).$ Consider indices $(k,j),(k\,',j\,')\in \mathbb{N}\times J.$ The entry $\tilde{a}_{(k,j),(k\,',j\,')}$ of the matrix $\tilde{a}$ is the $(j,j\,')$-th entry of the block $\tilde{a}_{k,k\,'}.$ We have $$\tilde{a}_{k,k\,'} =\sum_{t=0}^{k-2} a_{k-1-t,k\,' -t}. $$ Hence, $$\tilde{a}_{(k,j),(k\,',j\,')} =(\tilde{a}_{k,k\,'})_{j,j\,'} =\sum_{t=0}^{k-2} (a_{k-1-t,k\,' -t})_{j,j\,'}= \sum_{t=0}^{k-2} a_{(k-1-t,j),(k\,'-t,j\,')}.$$
We showed that the $(k,j)$-th row of the matrix $\tilde{a}$ is a sum of permuted $(1,j)$-th, $\ldots,$ $(k-1,j)$-th rows of the matrix $a.$ Hence, every row of the matrix $\tilde{a}$ contains finitely many nonzero elements, $\tilde{a}\in M_{rcf}(I,\mathbb{F}).$ This proves the part (b) of the theorem.

\vspace{10pt}
(c) Consider a $(\mathbb{Z}\times \mathbb{Z})$-matrix $$E_J =\sum_{i\in \mathbb{Z}} e_{i,i+1}(1).$$ For a $(\mathbb{Z}\times \mathbb{Z})$-matrix  $a=(a_{ij})_{i,j\in \mathbb{Z}}$ consider  the matrix $\tilde{a}=(\tilde{a}_{ij})_{i,j\in \mathbb{Z}},$ where $$ \tilde{a}_{ij} = \begin{cases}
                                                                                       a_{i-1,j}+ \cdots + a_{1,2-i+j} & \hbox{for $i\geq 2$} \\
                                                                                       0 & \hbox{for $i=1$} \\
                                                                                       -a_{i,j+1}- \cdots - a_{0,j-i+1} & \hbox{for $i\leq 0$.}
                                                                                     \end{cases} $$
We claim that $[E_J,\tilde{a}]=a.$ Indeed, a straightforward computation shows that
\begin{enumerate}
  \item[$(1)$] $[E_J,\tilde{a}]_{ij}=\tilde{a}_{i+1,j}-\tilde{a}_{i,j+1},$
  \item[$(2)$] $\tilde{a}_{i+1,j}-\tilde{a}_{i,j+1}=a_{ij}$
\end{enumerate}
for all $i,j\in \mathbb{Z}.$

It remains to check that for an arbitrary matrix $a\in MJ(\mathbb{F})$ the matrix $\tilde{a}$ also lies in $MJ(\mathbb{F}).$ Suppose that $a_{ij}=0$ whenever $|i-j|>k.$ The expression for $\tilde{a}_{ij}$ implies that $\tilde{a}_{ij}=0$ whenever $|i-j|>k+1.$ Hence $\tilde{a}\in MJ(\mathbb{F}).$ This completes the proof of the  part~(c).

\vspace{10pt}
(d) The Lie algebra $\mathfrak{o}_{\infty}(I,\mathbb{F})$ (resp., $\mathfrak{sp}_{\infty}(I,\mathbb{F})$) is simple since it has a local system of simple finite-dimensional subalgebras $\mathfrak{o}_{\infty}(I_0,\mathbb{F})$ (resp., $\mathfrak{sp}_{\infty}(J_0,\mathbb{F})$), where $I_0,$ $J_0$ run over all nonempty finite subsets of $I,$ the order of $J_0$ is even. This implies $\mathfrak{o}_{\infty}(I,\mathbb{F})=[\mathfrak{o}_{\infty}(I,\mathbb{F}),\mathfrak{o}_{\infty}(I,\mathbb{F})]$ and $\mathfrak{sp}_{\infty}(I,\mathbb{F})=[\mathfrak{sp}_{\infty}(I,\mathbb{F}),\mathfrak{sp}_{\infty}(I,\mathbb{F})]$. \end{proof}

In \cite{Herstein}, I.N.~Herstein formulated a series of conjectures about links between derivations and automorphisms of an associative algebra $A$ and derivations and automorphisms of Lie algebras $[A,A],$ $K(A,*).$ These conjectures were proved by K.~Beidar, M.~Bre\v{s}ar, M.~Chebotar and W.~Martindale in \cite{Bei_Bre_Cheb_Mart_1,Bei_Bre_Cheb_Mart_2,Bei_Bre_Cheb_Mart_3}.

We will formulate here and in Sec.~\ref{Aut_par} only particular cases of their results that are directly related to this work:

\vspace{10pt}
\hspace{-12pt}\textbf{Theorem I (K.I.~Beidar, \ M.~Bre\v{s}ar, \ M.~Chebotar \ and W.S.~Martindale; see, Cor. 1.4(b), \cite{Bei_Bre_Cheb_Mart_3})}  Let $A$ be a simple associative algebra with the center $Z.$ The algebra $A$ is not finite-dimensional over $Z.$ Let $d$ be a derivation of the Lie algebra $[A,A] \diagup [A,A]\cap Z.$ Then there exists a derivation $\tilde{d}:A\rightarrow A$ of the associative algebra $A$ such that $d(a)-\tilde{d}(a)\in Z$ for an arbitrary element $a\in [A,A].$

\vspace{10pt}
\hspace{-12pt}\textbf{Theorem II (K.I.~Beidar, \ M.~Bre\v{s}ar, \ M.~Chebotar \ and W.S.~Martindale; see, Cor. 1.9(b), \cite{Bei_Bre_Cheb_Mart_3})} Let $A$ be a simple associative algebra with an involution $*:A\rightarrow A.$ Let $Z$ be the center of the algebra $A.$ The algebra $A$ is not finite-dimensional over $Z.$ Let $d$ be a derivation of the Lie algebra $[K,K] \diagup [K,K]\cap Z,$ where $K=K(A,*)=\{a\in A\, | \, a^{*}=-a\}.$ Then there exists a derivation $\tilde{d}:A\rightarrow A$ of the associative algebra $A$ such that $d(a)-\tilde{d}(a)\in Z$ for an arbitrary element $a\in [K,K].$

\begin{proof}[Proof of Theorem $\ref{Th_IN_M_2_NEW}$] (a) The algebra $A=M_{\infty}(I,\mathbb{F})$ is simple and has zero center. Let $d$ be a derivation of the Lie algebra $\mathfrak{sl}_{\infty}(I,\mathbb{F})=[A,A].$ By Theorem I above, there exists a derivation $\tilde{d}$ of the algebra $A$ that coincides with $d$ on $[A,A].$  By Theorem \ref{Th_IN_M_1_NEW}(a), the derivation $\tilde{d}$ looks as  $\text{ad}_{\mathfrak{sl}_{\infty}(I,\mathbb{F})}(a),$ where $a\in M_{rcf}(I,\mathbb{F}).$

\vspace{10pt}
(b) Now, let $*$ be the transpose involution or the symplectic involution of the algebra $A=M_{\infty}(I,\mathbb{F}).$ Then  $\mathfrak{o}_{\infty}(I,\mathbb{F})$ (resp., $\mathfrak{sp}_{\infty}(I,\mathbb{F})$) is the Lie algebra of skew-symmetric elements $K=K(A,*).$ By Theorem II above, there exists a  derivation $\tilde{d}$ of the algebra $A$ that coincides with $d$ on $[K,K].$  By Theorem \ref{Pr_1_IN_M}$(d),$ $[K,K]=K.$ Theorem \ref{Th_IN_M_1_NEW}(a) implies that the  derivation $\tilde{d}$ is the restriction of a derivation   $\text{ad}_{M_{\infty}(I,\mathbb{F})}(a),$ where $a\in M_{rcf}(I,\mathbb{F}).$

Both the transpose and the symplectic involution on the algebra $M_{\infty}(I,\mathbb{F})$ extend to $M_{rcf}(I,\mathbb{F}).$ Let $$a=a_h + a_k, \quad \text{where} \quad a_h^{*}=a_h,\quad a_k^{*}=-a_k.$$ For an arbitrary element $b\in K$ we have $  [a,b]=[a_h ,b]+ [a_k , b]\in K.$ However, $$ [a_h,b]^{*}=[a_h ,b], \quad [a_k , b]^{*}=-[a_k,b].$$ Hence $[a_h,K]=(0),$ $[a,b]=[a_k,b],$ that is, $\text{ad}_K(a)=\text{ad}_K(a_k),$ where $a_k\in K(M_{rcf}(I,\mathbb{F}),*).$

\vspace{10pt}
(c) Let $d$ be a derivation of the Lie algebra $\mathfrak{gl}(I,\mathbb{F})=M(I,\mathbb{F})^{(-)}.$ By Theorem I of Beidar-Bre\v{s}ar-Chebotar-Martindale \cite{Bei_Bre_Cheb_Mart_3}, there exists a derivation $\tilde{d}$  of the  associative algebra $M(I,\mathbb{F})$ such that $d(a)-\tilde{d}(a)$ lies in the center $Z$ of $M(I,\mathbb{F})$ for all elements $a\in [M(I,\mathbb{F}),M(I,\mathbb{F})].$ By Theorem \ref{Pr_1_IN_M}(a),  $[M(I,\mathbb{F}),M(I,\mathbb{F})]=M(I,\mathbb{F}),$ hence $d(a)-\tilde{d}(a)\in Z$ for all elements $a\in M(I,\mathbb{F}).$

Now, for arbitrary elements $a,b\in M(I,\mathbb{F})$  we have $$d([a,b])=[d(a),b] + [a, d(b)]= [\tilde{d}(a),b]+[a,\tilde{d}(b)]=\tilde{d}([a,b]).$$ Hence, $d=\tilde{d}$ on  $[M(I,\mathbb{F}),M(I,\mathbb{F})].$ Again by Theorem  \ref{Pr_1_IN_M}(a), we get $d=\tilde{d}.$ It remains to refer to Theorem \ref{Th_IN_M_1_NEW}(c). This proves the assertion about derivations of  $\mathfrak{gl}(I,\mathbb{F}).$

Let $d$ be a derivation of the Lie algebra $\mathfrak{gl}_{rcf}(I,\mathbb{F}).$ The subalgebra $\mathfrak{sl}_{\infty}(I,\mathbb{F})$ is an ideal of $\mathfrak{gl}_{rcf}(I,\mathbb{F}).$ For arbitrary elements $a,b\in \mathfrak{sl}_{\infty}(I,\mathbb{F})$ we have $$d([a,b])=[d(a),b] + [a,d(b)]\in [\mathfrak{gl}_{rcf}(I,\mathbb{F}), \mathfrak{sl}_{\infty}(I,\mathbb{F})]\subseteq \mathfrak{sl}_{\infty}(I,\mathbb{F}).$$ Hence the ideal $\mathfrak{sl}_{\infty}(I,\mathbb{F})$ is differentially invariant. By Theorem \ref{Th_IN_M_2_NEW}$(a),$ there exists an element $a\in \mathfrak{gl}_{rcf}(I,\mathbb{F})$ such that $d(x)=[a,x]$ for an arbitrary element $x\in \mathfrak{sl}_{\infty}(I,\mathbb{F}).$ Consider the derivation $d\,'= d- \text{ad}(a)$ of the algebra $\mathfrak{gl}_{rcf}(I,\mathbb{F}).$ We have $$d\,'(\mathfrak{sl}_{\infty}(I,\mathbb{F})) =(0).$$ For arbitrary elements $x\in \mathfrak{sl}_{\infty}(I,\mathbb{F}),$ $y\in \mathfrak{gl}_{rcf}(I,\mathbb{F})$ we have $$ 0= d\,'([x,y])=[x,d\,'(y)].$$ Hence $d\,' (\mathfrak{gl}_{rcf}(I,\mathbb{F}))$ lies in the centralizer of  $\mathfrak{sl}_{\infty}(I,\mathbb{F}).$ It is easy to see that the centralizer of  $\mathfrak{sl}_{\infty}(I,\mathbb{F})$ in the algebra  $\mathfrak{gl}_{rcf}(I,\mathbb{F})$ is the space of scalar matrices $\alpha\cdot \text{Id},$ where $\alpha\in \mathbb{F}$ and $\text{Id}$ is the identity $(I\times I)$-matrix. Hence $$ d\,' (\mathfrak{gl}_{rcf}(I,\mathbb{F}))\subseteq \mathbb{F} \cdot \text{Id}.$$ For arbitrary elements $a,b\in \mathfrak{gl}_{rcf}(I,\mathbb{F})$ we have $d\,'([a,b])=[d\,'(a),b] + [a,d\,'(b)]=0. $ Hence,  $$ d\,' ([\mathfrak{gl}_{rcf}(I,\mathbb{F}),\mathfrak{gl}_{rcf}(I,\mathbb{F})])=(0).$$ By Theorem \ref{Pr_1_IN_M}$(a),$ $$ [\mathfrak{gl}_{rcf}(I,\mathbb{F}),\mathfrak{gl}_{rcf}(I,\mathbb{F})]=\mathfrak{gl}_{rcf}(I,\mathbb{F}),$$ which implies $d\,' =0,$ $d=\text{ad}(a).$ This completes the proof of the part concerning $\mathfrak{gl}_{rcf}(I,\mathbb{F}).$

Let $d:\mathfrak{gl}_{J}(\mathbb{F})\rightarrow \mathfrak{gl}_{J}(\mathbb{F})$ be a derivation. As above, we conclude that the ideal  $\mathfrak{sl}_{\infty}(\mathbb{Z},\mathbb{F})$ of the algebra $\mathfrak{gl}_{J}(\mathbb{F})$ is differentially invariant. By Theorem \ref{Th_IN_M_2_NEW}$(a),$ there exists a matrix $a\in M_{rcf}(\mathbb{Z},\mathbb{F})$ such that $d(x)=[a,x]$ for an arbitrary element $x\in \mathfrak{sl}_{\infty}(\mathbb{Z},\mathbb{F}).$ Consider the bimodule derivation $$d\,' :\mathfrak{gl}_{J}(\mathbb{F})\rightarrow \mathfrak{gl}_{rcf}(\mathbb{Z},\mathbb{F}), \quad d\,'(x)=d(x)-[a,x], \quad x\in \mathfrak{gl}_{J}(\mathbb{F}).$$ We have $d\,'(\mathfrak{sl}_{\infty}(\mathbb{Z},\mathbb{F})) =(0).$ As above, for arbitrary elements $x\in \mathfrak{sl}_{\infty}(\mathbb{Z},\mathbb{F}),$ $y\in \mathfrak{gl}_{J}(\mathbb{F})$ we have $$ 0= d\,'([x,y])=[x,d\,'(y)].$$ Hence $d\,' (\mathfrak{gl}_{J}(\mathbb{F}))$ lies in the centralizer of  $\mathfrak{sl}_{\infty}(\mathbb{Z},\mathbb{F})$ in the algebra $\mathfrak{gl}_{rcf}(\mathbb{Z},\mathbb{F}).$ In other words,  $$ d\,' (\mathfrak{gl}_{J}(\mathbb{F}))\subseteq \mathbb{F} \cdot \text{Id},$$ where $\text{Id}$ is the identity $(\mathbb{Z}\times \mathbb{Z})$-matrix. By Theorem \ref{Pr_1_IN_M}(a), $\mathfrak{gl}_{J}(\mathbb{F})=[\mathfrak{gl}_{J}(\mathbb{F}),\mathfrak{gl}_{J}(\mathbb{F})]$. So, by Proposition \ref{Pr_1_IN_M}(2), we conclude that  $d\,' =0.$ Hence  $d(x)=[a,x]$ for an arbitrary element $x\in \mathfrak{gl}_{J}(\mathbb{F}). $

It was already mentioned in the proof of Theorem \ref{Th_IN_M_1_NEW}(c) that if a $(\mathbb{Z}\times \mathbb{Z})$-matrix $a$ does not lie in $\mathfrak{gl}_{J}(\mathbb{F})$ then the matrix $[a,E_1],$ $E_1=\sum_{i\in \mathbb{Z}} e_{i,i+1}(1),$ also does not lie in $\mathfrak{gl}_{J}(\mathbb{F}).$ Hence $a\in \mathfrak{gl}_{J}(\mathbb{F}),$ which completes the proof of Theorem \ref{Th_IN_M_2_NEW}. \end{proof}

\section{Automorphisms of algebras of infinite matrices}\label{Aut_par}

Let $V$ be an infinite-dimensional vector space and let $V^{*}$  be the dual space of $V.$ A subspace $W\subseteq V^{*}$  is called \textit{total} if $$ \bigcap_{w\in W} \ker w =(0).$$

Recall the action of the algebra $\text{End}_{\mathbb{F}}(V)$ on the dual space $V^{*}$: for a linear transformation $a\in\text{End}_{\mathbb{F}}(V)$ and a linear functional $f\in V^{*}$ we define $(fa)(v)=f(a(v)),$ $v\in V.$

For a total subspace $W\subseteq V^{*}$ consider the subalgebra  $\text{End}_{\mathbb{F}}(V|W)=\{a\in \text{End}_{\mathbb{F}}(V) \, | \, Wa \subseteq W\}.$  Then the subalgebra  $$\text{End}_{fin}(V|W)=\text{End}_{\mathbb{F}}(V|W)\cap \text{End}_{fin}(V)$$  is a locally matrix algebra that can be identified with the tensor product $ V \otimes_{\mathbb{F}} W.$ Moreover, $\text{End}_{fin}(V|W)$ is  dense in $ \text{End}_{\mathbb{F}}(V)$ in the Tykhonoff topology (see, \cite{Jacobson}).

Let $\mathcal{E}=\{e_i, i\in I\}$ be a basis of the vector space $V$ and let $\mathcal{E}^{*}=\{e_i^{*}, i\in I\}\subset V^{*}$ be the dual basis: $e_i^{*}(e_j)=\delta_{ij}e_j.$ It is easy to see that the subspace $\text{Span}_{\mathbb{F}}(\mathcal{E}^{*})$ is total. In the basis $\mathcal{E}$ the algebra $ \text{End}_{\mathbb{F}}(V|\text{Span}_{\mathbb{F}}(\mathcal{E}^{*}))$ corresponds to the matrix algebra $M_{rcf}(I,\mathbb{F}).$ The subalgebra $ \text{End}_{fin}(V|\text{Span}_{\mathbb{F}}(\mathcal{E}^{*}))$ corresponds to  $M_{\infty}(I,\mathbb{F}).$

If $W=V^{*}$ then $\text{End}_{\mathbb{F}}(V|W)=\text{End}_{\mathbb{F}}(V),$ $\text{End}_{fin}(V|W)=\text{End}_{fin}(V).$

N.~Jacobson (\cite{Jacobson}, Chap.~9, Sec.~11, Th.~7) proved that for an arbitrary algebra $A,$ $$ \text{End}_{fin}(V|W)\subseteq A \subseteq \text{End}_{\mathbb{F}} (V|W)$$ and an arbitrary automorphism $\varphi$ of $A$ there exists an element $x\in G(\text{End}_{\mathbb{F}}(V|W))$ such that $\varphi(a)=x^{-1}ax,$ $a\in A.$

Substituting $W=\text{Span}_{\mathbb{F}}(\mathcal{E}^{*})$ or $W=V^{*}$  we obtain Theorem \ref{Th_IN_M_3_NEW}.

In the proof of Theorem \ref{Th_IN_M_5_NEW} on automorphisms of Lie algebras, we will again rely on the work \cite{Bei_Bre_Cheb_Mart_3} of K.I.~Beidar, M.~Bre\v{s}ar, M.~Chebotar and W.S.~Martindale on Herstein's Conjectures.

The results in \cite{Bei_Bre_Cheb_Mart_3} are quite general. As in Sec.~\ref{Lie algebras}, we will use only a small part of them that is directly related to our proof.

\vspace{10pt}
\hspace{-12pt}\textbf{Theorem III (K.I.~Beidar, \ M.~Bre\v{s}ar, \ M.~Chebotar \ and W.S.~Martindale; see, Cor. 1.2 (c), \cite{Bei_Bre_Cheb_Mart_3})}\label{BBCM1} Let $\emph{char}\,\mathbb{F}\neq 2$ and let $A$ be a simple associative $\mathbb{F}$-algebra with the center $Z$ and $A$ is not  finite-dimensional over $Z.$ If $\alpha$ is an automorphism of the Lie algebra $[A,A]$ then there exist mappings $\varphi, \psi:A\rightarrow A$ such that
\begin{enumerate}
  \item[$(i)$] $\varphi$ is either an automorphism or the negative of an anti-automor\-phism of $A,$
  \item[$(ii)$] $\psi(A)\subseteq Z,$ $\psi([A,A])=(0),$ 
\end{enumerate}
and $\alpha(a)=\varphi(a)+\psi(a)$ for all $a\in [A,A].$

\vspace{10pt}
\hspace{-12pt}\textbf{Theorem IV (K.I.~Beidar, \ M.~Bre\v{s}ar, \ M.~Chebotar \ and W.S.~Martindale; see, Cor. 1.7, \cite{Bei_Bre_Cheb_Mart_3})}\label{BBCM2} Let $\emph{char}\,\mathbb{F}\neq 2$ and Let $A$ be a simple associative $\mathbb{F}$-algebra with an involution and let $K$ be the Lie algebra of skew-symmetric elements. Suppose  that the algebra $A$ is not finite-dimensional over $Z.$ Then for an arbitrary automorphism $\alpha$ of the Lie algebra $[K,K]$ there exists an automorphism $\varphi$ of the algebra $A$ such that $\alpha(a)=\varphi(a) $ for all $a\in [K,K].$

\vspace{10pt}
Theorem III above implies that in order to describe Lie   automorphisms of a simple associative algebra we need to describe its anti-automorphisms first.

\begin{proof}[Proof of Theorem $\ref{Th_IN_M_4_NEW}$] (a) The algebras $M_{\infty}(I,\mathbb{F}),$  $M_{rcf}(I,\mathbb{F}),$  $MJ(\mathbb{F})$ are closed with respect to the transpose $t.$ If $*$ is an anti-automorphism then the composition $\varphi$ of $*$ and $t$ is an automorphism, $(a^t)^{*}=\varphi(a)$ for an arbitrary element $a.$ Hence, $a^{*}=\varphi(a^t).$

\vspace{10pt}
(b) N.~Jacobson (\cite{Jacobson}, Chap. IX.12, Th. 8) showed that if an algebra $A,$ $\text{End}_{fin}(V|W)\subseteq A \subseteq \text{End}(V|W),$ has an anti-isomorphism then $\dim_{\mathbb{F}}V = \dim_{\mathbb{F}} W.$ P.~Erdos and J.~Kaplansky (see, \cite{Jacobson}, Chap. IX.5, Th. 2) showed that $$\dim_{\mathbb{F}}V^{*}= |\mathbb{F}|^{\dim_{\mathbb{F}}V}> \dim_{\mathbb{F}}V.$$ Hence, the algebras $$\text{End}_{fin}(V)= \text{End}_{fin}(V|V^{*}) \quad \text{and} \quad M(I,\mathbb{F}) \cong \text{End}_{\mathbb{F}}(V)=\text{End}(V|V^{*})$$ do not have anti-isomorphisms.  \end{proof}

\begin{proof}[Proof of Theorem $\ref{Th_IN_M_5_NEW}$]  The part (a) immediately follows from Theorem~III above applied to the algebra $A=M_{\infty} ( I,\mathbb{F})$ and  Theorem~\ref{Th_IN_M_3_NEW}(a).

\vspace{10pt}
(b) Let $ ^{*} $ be the transpose or the symplectic involution on the algebra $A=M_{\infty}  ( I,\mathbb{F} ).$ Then $K = K (A,^{*})=\mathfrak{o}_{\infty}(I,\mathbb{F})$ or $\mathfrak{sp}_{\infty}(I,\mathbb{F}).$ By  Theorem \ref{Th_IN_M_1(1)}(d), we have $K=[ K,K].$ By Theorem~IV above, an arbitrary automorphism $\varphi$ of the Lie algebra $K$ lifts to an automorphism of the algebra $A.$ By  Theorem \ref{Th_IN_M_3_NEW}$(a),$ there exists an element $x \in G (M_{rcf} ( I,\mathbb{F} ) )$ such that $\varphi ( a ) = x^{-1}ax$ for all elements $a \in K.$ Hence $$ ( x^{-1}ax)^{*}=-x^{*}a (x^{*})^{-1}=-x^{-1}ax,$$ which implies $$ ( xx^{*}) a  ( xx^{*})^{-1}=a.$$

We proved that the element $xx^{*}$ commutes with all elements from $K.$ Since the algebra $A$ is generated by the subspace $K$ it follows that the element $xx^{*}$ commutes with all elements from $M_{\infty} ( I,F )$ and, therefore, is a scalar matrix $xx^{*}=\alpha\cdot \text{Id},$ $0\neq \alpha \in \mathbb{F}.$ This completes the proof of the part (b).

\vspace{10pt}
The part (c) immediately follows from Theorem~III above,  Theorem~\ref{Th_IN_M_1(1)}, Theorem~\ref{Th_IN_M_3_NEW}(c), and  parts (d),(e) follow from Theorem~IV above,  Theorem~\ref{Th_IN_M_4_NEW}(b) and Theorem~\ref{Th_IN_M_3_NEW}(b),(c). This completes the proof of  Theorem \ref{Th_IN_M_5_NEW}.\end{proof}

\vspace{15pt}
\address{\small Faculty of Mechanics and Mathematics, Taras Shevchenko National University of Kyiv, Volodymyrska, 60, Kyiv 01033, Ukraine}

\email{\small \emph{Email address}: bezusch@univ.kiev.ua}

\end{document}